\documentclass[12pt]{article}
\usepackage{amssymb,bm}
\usepackage{amsmath}
\usepackage{amsthm}
\usepackage{graphicx}
\usepackage{hyperref} 
\hypersetup{pdfborder=0 0 0}
\newtheorem{theorem}{{\sc Theorem}}[section]

\newtheorem{lemma}[theorem]{{\sc Lemma}}
\newtheorem{corollary}[theorem]{Corollary}
\newtheorem{remark}[theorem]{Remark}

\newcommand{\bb}[1]{\mathbb{ #1}}

\bmdefine\Bone{1}

\newcommand{\weak}{\rightharpoonup\:}

\newcommand{\eqv}{\Longleftrightarrow}
\newcommand{\bra}[1]{\overline{#1}}

\newcommand{\hf}{\displaystyle\frac{1}{2}}
\newcommand{\nth}[1]{\displaystyle\frac{1}{#1}}

\renewcommand{\Hat}[1]{\widehat{#1}}

\def\XXint#1#2#3{{\setbox0=\hbox{$#1{#2#3}{\int}$ }
\vcenter{\hbox{$#2#3$ }}\kern-.6\wd0}}

\newcommand{\rhs}{right-hand side}
\newcommand{\lhs}{left-hand side}

\newcommand{\IFF}{if and only if }

\newcommand{\Ga}{\alpha}
\newcommand{\Gb}{\beta}

\newcommand{\Ge}{\epsilon}

\newcommand{\Gg}{\gamma}

\newcommand{\Gl}{\lambda}

\newcommand{\Go}{\omega}

\newcommand{\Gz}{\zeta}

\newcommand{\GG}{\Gamma}

\newcommand{\GO}{\Omega}

\bmdefine\BGa{\alpha}
\bmdefine\BGb{\beta}
\bmdefine\BGd{\delta}
\bmdefine\BGe{\epsilon}
\bmdefine\BGve{\varepsilon}
\bmdefine\BGf{\phi}
\bmdefine\BGvf{\varphi}
\bmdefine\BGg{\gamma}
\bmdefine\BGc{\chi}
\bmdefine\BGi{\iota}
\bmdefine\BGk{\kappa}
\bmdefine\BGl{\lambda}
\bmdefine\BGn{\eta}
\bmdefine\BGm{\mu}
\bmdefine\BGv{\nu}
\bmdefine\BGp{\pi}
\bmdefine\BGth{\theta}
\bmdefine\BGvth{\vartheta}
\bmdefine\BGr{\rho}
\bmdefine\BGvr{\varrho}
\bmdefine\BGs{\sigma}
\bmdefine\BGvs{\varsigma}
\bmdefine\BGt{\tau}
\bmdefine\BGj{\tau}
\bmdefine\BGu{\upsilon}
\bmdefine\BGo{\omega}
\bmdefine\BGx{\xi}
\bmdefine\BGy{\psi}
\bmdefine\BGz{\zeta}
\bmdefine\BGD{\Delta}
\bmdefine\BGF{\Phi}
\bmdefine\BGG{\Gamma}
\bmdefine\BGL{\Lambda}
\bmdefine\BGP{\Pi}
\bmdefine\BGT{\Theta}
\bmdefine\BGS{\Sigma}
\bmdefine\BGU{\Upsilon}
\bmdefine\BGO{\Omega}
\bmdefine\BGX{\Xi}
\bmdefine\BGY{\Psi}

\bmdefine\BFM{\mathfrak{M}}
\bmdefine\BFb{\mathfrak{b}}
\bmdefine\BFk{\mathfrak{k}}
\bmdefine\BFm{\mathfrak{m}}
\bmdefine\BFu{\mathfrak{u}}
\bmdefine\BFv{\mathfrak{v}}

\newcommand{\CK}{{\mathcal K}}

\newcommand{\CR}{{\mathcal R}}

\bmdefine\BCA{{\mathcal A}}
\bmdefine\BCB{{\mathcal B}}
\bmdefine\BCC{{\mathcal C}}
\bmdefine\BCD{{\mathcal D}}
\bmdefine\BCE{{\mathcal E}}
\bmdefine\BCF{{\mathcal F}}
\bmdefine\BCG{{\mathcal G}}
\bmdefine\BCH{{\mathcal H}}
\bmdefine\BCI{{\mathcal I}}
\bmdefine\BCJ{{\mathcal J}}
\bmdefine\BCK{{\mathcal K}}
\bmdefine\BCL{{\mathcal L}}
\bmdefine\BCM{{\mathcal M}}
\bmdefine\BCN{{\mathcal N}}
\bmdefine\BCO{{\mathcal O}}
\bmdefine\BCP{{\mathcal P}}
\bmdefine\BCQ{{\mathcal Q}}
\bmdefine\BCR{{\mathcal R}}
\bmdefine\BCS{{\mathcal S}}
\bmdefine\BCT{{\mathcal T}}
\bmdefine\BCU{{\mathcal U}}
\bmdefine\BCV{{\mathcal V}}
\bmdefine\BCW{{\mathcal W}}
\bmdefine\BCX{{\mathcal X}}
\bmdefine\BCY{{\mathcal Y}}
\bmdefine\BCZ{{\mathcal Z}}

\bmdefine\Bzr{ 0}
\bmdefine\Ba{ a}
\bmdefine\Bb{ b}
\bmdefine\Bc{ c}
\bmdefine\Bd{ d}
\bmdefine\Be{ e}
\bmdefine\Bf{ f}
\bmdefine\Bg{ g}
\bmdefine\Bh{ h}
\bmdefine\Bi{ i}
\bmdefine\Bj{ j}
\bmdefine\Bk{ k}
\bmdefine\Bl{ l}
\bmdefine\Bm{ m}
\bmdefine\Bn{ n}
\bmdefine\Bo{ o}
\bmdefine\Bp{ p}
\bmdefine\Bq{ q}
\bmdefine\Br{ r}
\bmdefine\Bs{ s}
\bmdefine\Bt{ t}
\bmdefine\Bu{ u}
\bmdefine\Bv{ v}
\bmdefine\Bw{ w}
\bmdefine\Bx{ x}
\bmdefine\By{ y}
\bmdefine\Bz{ z}
\bmdefine\BA{ A}
\bmdefine\BB{ B}
\bmdefine\BC{ C}
\bmdefine\BD{ D}
\bmdefine\BE{ E}
\bmdefine\BF{ F}
\bmdefine\BG{ G}
\bmdefine\BH{ H}
\bmdefine\BI{ I}
\bmdefine\BJ{ J}
\bmdefine\BK{ K}
\bmdefine\BL{ L}
\bmdefine\BM{ M}
\bmdefine\BN{ N}
\bmdefine\BO{ O}
\bmdefine\BP{ P}
\bmdefine\BQ{ Q}
\bmdefine\BR{ R}
\bmdefine\BS{ S}
\bmdefine\BT{ T}
\bmdefine\BU{ U}
\bmdefine\BV{ V}
\bmdefine\BW{ W}
\bmdefine\BX{ X}
\bmdefine\BY{ Y}
\bmdefine\BZ{ Z}

\DeclareMathAlphabet{\pazocal}{OMS}{zplm}{m}{n}
\newcommand{\PH}{\pazocal{H}}
\newcommand{\PP}{\pazocal{P}}
\newcommand{\PW}{\pazocal{W}}
\DeclareMathAlphabet{\cg}{OMS}{zplm}{m}{n}
\newcommand{\HH}{\mathbb{H}}
\newcommand{\NN}{\mathbb{N}} 
\newcommand{\RR}{\mathbb{R}}      
\newcommand{\ZZ}{\mathbb{Z}}      
\newcommand{\D}{\text{d}}
\usepackage{xcolor}

\begin{document}


\title{Explicit power laws in analytic continuation problems via reproducing kernel Hilbert spaces.}
  
\author{Yury Grabovsky \and Narek Hovsepyan}
\date{}
\maketitle

\begin{abstract}
The need for analytic continuation arises frequently in the context of inverse problems. Notwithstanding the uniqueness theorems, such problems are notoriously ill-posed without additional regularizing constraints. We consider several analytic continuation problems with typical global boundedness constraints that restore well-posedness. We show that all such problems exhibit a power law precision deterioration as one moves away from the source of data. In this paper we demonstrate the effectiveness of our general Hilbert space-based approach for determining these exponents. The method identifies the ``worst case'' function as a solution of a linear integral equation of Fredholm type. In special geometries, such as the circular annulus or upper half-plane this equation can be solved explicitly. The obtained solution in the annulus is then used to determine the exact power law exponent for the analytic continuation from an interval between the foci of an ellipse to an arbitrary point inside the ellipse. Our formulas are consistent with results obtained in prior work in those special cases when such exponents have been determined.
\end{abstract}


\section{Introduction}
\setcounter{equation}{0} 
\label{sec:intro}
Many inverse problems reduce to analytic continuation questions when solutions
of direct problems are known to possess analyticity in a domain in the complex
plane but can be measured only on a subset (often a part of the boundary) of
this domain. For example, if one wants to recover a signal corrupted by a
low-pass convolution filter, then one needs to recover an entire function from
its measured values on an interval \cite{digr01,bdmh19}.  Another large class
of inverse problems can be termed ``Dehomogenization'' \cite{chou08,ou14}, where
one wants to reconstruct some details of microgeometry from measurements of
effective properties of the composite. The idea of reconstruction is based on
the analytic properties of effective moduli \cite{berg,mi81,gopa83} of
composites. See e.g. \cite{mchcg15} for an extensive bibliography in this area.

The method of recovery via analytic continuation is a tempting proposition in
view of the uniqueness properties of analytic functions. Unfortunately,
analyticity is a local property ``stored'' at an infinite depth within the
continuum of function values and can be represented by delicate cancellation
properties responsible for the validity of Carleman and Carleman type
extrapolation formulas \cite{carl26,gokr33,aize12}. Adding small errors to the
exact values of analytic functions destroys these local properties. Instead we
want to accumulate the remnants of analyticity and use global properties of
analytic functions to achieve analytic continuation. This is only possible
under some additional regularizing constraints, such as global boundedness
\cite{doug60,cami65,stef86,chuyu,trefe19}. Taking this idea to the extreme, any
bounded entire function is a constant by Liouville's theorem, so that the
effect of boundedness depends strongly on the geometry of the domain of
analyticity. 

In order to quantify the degree to which analytic continuation is
possible, consider an analytic function $F$ in a domain $\Omega$. Assume that
$F$ is measured on a curve $\Gamma \Subset \Omega$ with a relative error
$\Ge$, with respect to some norm $\|F\|_\Gamma$. Can one perform an analytic
continuation of $F$ from $\GG$ to $\GO$ in the presence of measurement errors?
Without discussing specific analytic continuation algorithms we would like to
examine theoretical feasibility of such a procedure. For example, if two
different algorithms are deployed matching $F$ on $\GG$ with relative
precision $\Ge$ how far their outputs could possibly differ at a given point
$z\in\GO\setminus\GG$? To answer this question we consider the difference $f$
of the two purported analytic continuations. Such a difference will be small
on $\GG$, and we want to quantify how large such a function can possibly be at
some point $z\in\GO$ relative to its global size on $\GO$. 

Based on established upper and lower bounds, exact and numerical results
\cite{davis52,cami65,ciulli69,mill70,payne75,fran90,vese99,fdfd09,deto18,
trefe19} a general
\emph{power law principle} emerges, whereby the relative precision of analytic
continuation decays as power law $\Ge^{\Gg(z)}$, where the exponent
$\Gg(z)\in(0,1)$ decreases to 0, as we move further away from the source of
data. How fast $\Gg(z)$ decays depends strongly on the geometry of the domain
and the data source \cite{trefe19,grho-gen}. In \cite{grho-gen} we considered
an example, where $\GO$ is the complex upper half-plane and $\GG$ is the
interval $[-1,1]$ on the real axis. We have proved that for $z$ in the upper
half-plane $\Gg(z)$ is the angular size of the interval $[-1,1]$ as viewed
from $z$, measured in units of $\pi$. Conformal mappings can also be used to
relate the exponents for one geometry to the exponents for the conformally
equivalent ones. We believe that such power law transition from well-posedness
to practical ill-posedness is a general property of analytic continuation,
quantifying the tug-of-war between their rigidity (unique continuation
property) and flexibility (as in the Riesz density theorem \cite{part97}). 

The lower bounds on $\Gg(z)$ can be obtained by exhibiting bounded analytic
functions that are small on a curve $\GG$, but not quite as small at a
particular extrapolation point. The upper bounds are harder to prove but there
is ample literature where such results are achieved
\cite{davis52,cami65,ciulli69,mill70,payne75,fran90,vese99,fdfd09,deto18,
trefe19}. In fact, it
was observed in \cite{trefe19} that upper and lower bounds of the form
$\Ge^{\Gg(z)}$ on the extrapolation error do hold for all geometries. However,
with few exceptions the upper and lower bounds do not match. In those examples
where they do match \cite{deto18,trefe19} the optimality of the bounds are
concluded a posteriori.

In our recent work \cite{grho-gen} we have developed a new method for
characterizing analytic functions in the upper half-plane $\bb{H}_{+}$
attaining the optimal upper bound in terms of a solution of an integral
equation of the second kind with compact, positive, self-adjoint operator on
$L^{2}(\GG)$. In Section~\ref{SECT RKHS}, we extend this result to reproducing
kernel Hilbert spaces $\PH = \PH(\Omega)$ of analytic functions in a domain
$\GO\subset\bb{C}$. The error maximization problem is reformulated as a
maximization of a linear objective functional subject to quadratic
constraints, permitting us to use convex duality methods. The optimality
conditions take the form of a linear integral equation of Fredholm type, where
the positive, compact self-adjoint operator $\CK$ is expressed in terms of the
reproducing kernel of $\PH(\Omega)$. The integral operator $\CK$ occurs
frequently in the context of reproducing kernel Hilbert spaces
(e.g. \cite{davis52}) and is related to the restriction operator $\CR: \PH \to
L^2(\Gamma)$. Namely, $\CK = \CR^* \CR$. The exponent $\Gg(z)$ in the power
law asymptotics can then be expressed in terms of the rates of exponential
decay of eigenvalues of the integral operator $\CK$ and its eigenfunctions at
the extrapolation point $z\in\Omega$. For certain classes of restriction
operators the exponential decay of the eigenvalues of $\CK$ has been known for
a long time, and their exact asymptotics has been established in
\cite{parfenov} (see also \cite{zask76,parf78,gps03,pute17}).  Alternatively,
the exponent $\Gg(z)$ can be read off the explicit solution of the integral
equation in cases where such an explicit solution is available
\cite{grho-gen}. This allows us to compute $\Gg(z)$ explicitly in a number of
special cases. For example, when $\GG$ is a circle in the upper half-plane
(Section~\ref{SECT upper halfplane}) or a circle in an annulus
(Section~\ref{SECT annulus}).

In Section~\ref{SECT ellipse} we present a somewhat unexpected application of
the annulus result to the problem of analytic continuation in a Bernstein
ellipse \cite{bern12}, studied in \cite{deto18}. Since the annulus is not conformally
equivalent to the ellipse one would not expect a direct relation. The trick we
use, inspired by \cite{deto18}, is to map the Bernstein ellipse cut along
$[-1,1]$ onto the annulus using the inverse of the Joukowski function. Then,
functions analytic in the ellipse are distinguished from functions analytic in
the cut ellipse by their continuity across the cut. After the conformal
transformation the image of functions analytic in the entire ellipse would
consist of functions analytic in the annulus with a reflection symmetry on the
unit circle.  Our Hilbert space-based approach can easily incorporate linear
constraints by making an appropriate choice of the underlying Hilbert
space. However, the question is about the relation between the problems with
and without such constraints. In the case of the Bernstein ellipse and the
annulus, we discover that the subspace of functions analytic in the
annulus corresponding to functions analytic in the Bernstein ellipse is
invariant with respect to the integral operator $\CK$. It is this
invariance that permits us to solve the problem with additional linear
constraints using the known solution of the original problem. This is
discussed in Section~\ref{sub:lincon}. When the extrapolation point $z$ lies
on the real line inside the Bernstein ellipse we recover the optimal exponent
$\Gg(z)$ obtained in \cite{deto18}. However, our approach also gives the
formula for the exponent $\Gg(z)$ for arbitrary points $z$ inside the ellipse.

\section{Main Results}

\noindent \textbf{Notation:}  We will write $A
\lesssim B$, if there exists a constant $c$ such that $A \leq c B$ and
likewise the notation $A \gtrsim B$ will be used. If both $A
\lesssim B$ and $A \gtrsim B$ are satisfied, then we will write $A \simeq B$. Throughout the paper all the implicit constants will be independent of the parameter $\epsilon$.

\subsection{The annulus} \label{SECT annulus}

For $0<\rho <1$, $r>0$ let
\begin{equation} \label{A_rho and gamma}
\begin{split}
A_\rho = \{\zeta \in \mathbb{C} : \rho < |\zeta| < 1\},
\qquad \quad
\Gamma_r = \{\zeta \in \mathbb{C} : |\zeta| = r\}.
\end{split}
\end{equation}

\noindent Consider the Hardy space (e.g. \cite{duren})

\begin{equation} \label{H^2 annulus def}
H^2(A_\rho) = \{f \ \text{is analytic in} \ A_\rho: \|f\|_{H^2(A_\rho)}=\sup_{\rho < r < 1}  \|f\|_{L^2(\Gamma_r)}  < \infty \},
\end{equation}
where for a curve $\Gamma\subset\bb{C}$ the space $L^2(\Gamma)$ denotes the
space of square-integrable functions on $\GG$ with respect to the arc length measure
$|\D \tau|$ on $\Gamma$.
\begin{theorem}[Annulus] \label{THM annulus} Let $\Gamma = \Gamma_r$ with $r
  \in (\rho, 1)$ fixed and $z \in A_\rho \backslash \Gamma$. Then there exists
  $C>0$, such that for any $\Ge>0$ and any $f \in H^2(A_\rho)$ with
  $\|f\|_{H^2(A_\rho)} \leq 1$ and $\|f\|_{L^2(\Gamma)} \leq \epsilon$, we
  have
\begin{equation} \label{main bound annulus}
|f(z)| \leq C \epsilon^{\gamma(z)},
\end{equation}
where
\begin{equation} \label{gamma annulus}
\gamma(z) =
\begin{cases}
\dfrac{\ln |z|}{\ln r},  &\text{if} \ \ r<|z|<1
\\[3ex]
\dfrac{\ln (|z| / \rho)}{\ln (r/\rho)}, &\text{if} \ \ \rho<|z|<r
\end{cases}
\end{equation}
Moreover, \eqref{main bound annulus} is asymptotically optimal in $\epsilon$ and the function attaining the bound is
\begin{equation} \label{M annulus}
M(\zeta) = \epsilon^{2-\gamma(z)} \sum_{n \in \ZZ} \frac{(\overline{z}\zeta)^n}{r^{2n} + \epsilon^2 (1+\rho^{2n})}, \qquad \zeta \in A_\rho.
\end{equation}

\noindent In addition $M$ is analytic in the closure of $A_\rho$ and $\|M\|_{H^\infty (\overline{A_\rho})}$ is bounded uniformly in $\epsilon$.
\end{theorem}

\vspace{.1in}

\begin{remark}
The statement that $M$ attains the bound in \eqref{main bound annulus} means that $\|M\|_{H^2(A_\rho)} \lesssim 1$, \ $\|M\|_{L^2(\Gamma)} \lesssim \epsilon$ and $|M(z)| \simeq \epsilon^{\gamma(z)}$, with all implicit constants independent of $\epsilon$.
\end{remark}

It is somewhat surprising that the worst case function, which was required to
be analytic only in $A_{\rho}$ is in fact analytic in a larger annulus
$\{|z^*_\rho| < |\zeta| < |z^*_1| \}$, where $z^*_1 = 1/
\overline{z}$ is the point symmetric to $z$ w.r.t the circle $\Gamma_1$ and
$z^*_\rho = \rho^2 / \overline{z}$ is the point symmetric to $z$ w.r.t the
circle $\Gamma_\rho$. In particular, $M \in H^\infty (A_\rho)$. Hence,
$M(\zeta)$ also maximizes $|M(z)|$, asymptotically, as $\Ge\to 0$, if the
constraints were given in $H^{\infty}(A_\rho)$ and $L^{\infty}(\GG)$, instead of $H^{2}(A_\rho)$
and $L^{2}(\GG)$, respectively.

\begin{remark} \label{REM stability circle} The limiting case as $\rho \to 0$
  corresponds to the analytic continuation from the circle $\Gamma_r$ into the
  unit disk $D$. The limiting value of the exponent is $\gamma(z) = \frac{\ln
    |z|}{\ln r}$ for $|z|>r$, and $\gamma(z)=1$, for $|z|<r$. The numerical
  stability of extrapolation inside $\GG_{r}$ can be seen directly from
  Cauchy's integral formula. The same formula for $\Gg(z)$ has been obtained
  in \cite{trefe19} for $H^{\infty}(D)$.  
\end{remark}

\subsection{The upper half-plane} \label{SECT upper halfplane}

Let $\HH_+=\{z\in\bb{C}:\Im(z)>0\}$ denote the complex upper half-plane and consider the
Hardy space

\begin{equation*}
H^2(\HH_+) := \{f \ \text{is analytic in} \ \HH_+: \sup_{y>0} \|f(\cdot + iy)\|_{L^2(\RR)} < \infty \}.
\end{equation*}

\noindent It is well known \cite{koos98} that these functions have
$L^2$-boundary data, and that $\|f\|=\|f\|_{L^2(\RR)}$ defines a norm in
$H^2(\HH_+)$. Assume that the data curve $\Gamma \Subset \HH_+$ is a
circle. By considering affine automorphisms $z\mapsto az+b$, $a>0,\
b\in\bb{R}$, of $\HH_+$ we may ``translate'' $\Gamma$ to be centered at $i$.

\begin{theorem} \label{THM H+} Let $\GG$ be a circle centered at $i$ of radius
  $r<1$. Let $z\in\bb{H}_{+}$ be a point outside of $\GG$. Then there exists
  $C>0$, such that for any $\Ge>0$ and any $f \in H^2(\HH_+)$ with
  $\|f\|_{H^2(\HH_+)} \leq 1$ and $\|f\|_{L^2(\Gamma)} \leq \epsilon$, we have
\begin{equation} \label{main bound H+}
|f(z)|\leq C \Ge^{\Gg(z)},
\end{equation}

\noindent where

\begin{equation}
  \label{gammaUHP}
\Gg(z)=\frac{\ln|m(z)|}{\ln\rho},\qquad \qquad \rho=\frac{1-\sqrt{1-r^2}}{r},
\end{equation}

\noindent and 
\[
m(\zeta) = \frac{\zeta - z_0}{\zeta + z_0}, \qquad 
z_0 = i \sqrt{1-r^2}
\] 
is the M\"obius map transforming the upper half-plane into the unit disc and the circle $\GG$ into a concentric circle, whose radius has to be $\rho$. Moreover, \eqref{main bound H+} is asymptotically optimal in $\epsilon$ and the function attaining the bound can be written as a convergent in the upper half-plane ``power'' series

\begin{equation}
  \label{M upper halfplane}
  M(\Gz)=\frac{\epsilon^{2-\gamma(z)}}{\Gz+z_{0}}\sum_{n=1}^{\infty}\frac{\left(\bra{m(z)}m(\Gz)\right)^{n}}
{\Ge^{2}+\rho^{2n}},\qquad\Gz\in\bb{H}_{+}.
\end{equation}
\end{theorem}

\vspace{.1in}

\begin{remark} \label{REM stability in UHP}
When $z$ is inside $\Gamma$ we have complete stability, indeed Cauchy's integral formula implies that

$$|f(z)| \leq c \epsilon$$

\noindent for a constant $c$ independent on $\epsilon$.
\end{remark}

\subsection{The Bernstein ellipse}
Let $E_R$ be the open ellipse with foci at $\pm 1$ and the sum of semi-minor
and semi-major axes equal to $R>1$. The axes lengths of such an ellipse are
therefore $(R\pm R^{-1})/2$. $E_{R}$ is called the
Bernstein ellipse \cite{bern12,tref12}. Its boundary is an image of a circle of radius $R$ centered
at the origin under the Joukowski map $J(z)=(z+z^{-1})/2$.  Let $H^\infty
(E_R)$ be the space of bounded analytic functions in $E_R$, with the usual
supremum norm.

\begin{theorem}\label{THM ellipse}
Let $z \in E_R \backslash [-1,1]$. Then there exists $C>0$, such that for
every $\Ge>0$ and $F \in H^\infty (E_R)$ with
$\|F\|_{H^\infty (E_R)} \leq 1$ and $\|F\|_{L^\infty (-1,1)} \leq \Ge$,
we have
\begin{equation} \label{main bound ellipse}
|F(z)| \leq C \epsilon^{\alpha(z)},
\end{equation}

\noindent where

\begin{equation} \label{alpha}
\alpha(z) = 1 - \frac{\ln   \left| J^{-1}(z)  \right| }{\ln R} \in
(0,1),\qquad
J^{-1}(z) = z + (z-1)\sqrt{\frac{z + 1}{z - 1}}
\end{equation}

\noindent Moreover, \eqref{main bound ellipse} is asymptotically optimal in $\epsilon$ and function attaining the bound is

\begin{equation} \label{M ellipse}
M (\Gz) = \epsilon^{2 - \alpha(z)}  \sum_{n=1}^\infty \frac{(\overline{J^{-1} (z)})^nT_n(\Gz)}{1 + \epsilon^2 R^{2n}},
\end{equation}

\noindent where $T_n$ is the Chebyshev polynomial of degree $n$: $T_n(x) =
\cos (n \cos^{-1} x)$ for $x \in [-1,1]$.

\end{theorem}

Several remarks are now in order.

\begin{enumerate} 

\item[(i)] $J^{-1}(\zeta)$ is the branch of an inverse of the Joukowski map $J$, that is analytic in the slit ellipse $E_R \backslash [-1,1]$ and satisfies the inequalities $1 < |J^{-1}(\zeta)| <  R$.

\item[(ii)] Chebyshev polynomials $T_{n}$ play the same role in the ellipse as
  monomials $\Gz^{n}$ play in the annulus, i.e. they are the building blocks
  of analytic functions. In fact $J^{-1} \circ T_n \circ J = \zeta^n$.

\item[(iii)] The same bound (\ref{main bound ellipse}) was obtained
in \cite{deto18} when $z\in E_{R}\cap\bb{R}$, where it was shown that the 
bound (up to logarithmic factors) could be attained by a polynomial
\begin{equation}
\label{DT}
g(\zeta) = \Ge T_{K(\Ge)}(\zeta) , \qquad K=K(\Ge) =  \lfloor{\ln (1/\epsilon) / \ln R}\rfloor.
\end{equation}
We observe that the terms in (\ref{M ellipse}) increase exponentially fast
from $n=1$ to $n=K(\Ge)$ and then decrease exponentially fast for
$n>K(\Ge)$. Hence, asymptotically (up to logarithmic factors) we can say that
\[
|M(\Gz)|\approx\epsilon^{2 - \alpha(z)}\frac{\left|J^{-1}(z)\right|^{K(\Ge)}
|T_{K(\Ge)}(\Gz)|}{1 + \epsilon^2 R^{2K(\Ge)}}\approx\Ge |T_{K(\Ge)}(\Gz)|,
\]
in agreement with (\ref{DT}).
\end{enumerate}

\section{Quantifying stability of analytic continuation}
\setcounter{equation}{0}
 \label{SECT abstract theory}

\subsection{Reproducing kernel Hilbert spaces} \label{SECT RKHS}
Our goal is to characterize how large a function $f$ analytic in a domain
$\Omega$ can be at a point $z\in\GO$, provided that it is small on a curve
$\Gamma \Subset \Omega$, relative to its global size in $\GO$. If some norms
$\|f\|_\Gamma$ and $\|f\|_{\PH}$ are used to measure the magnitude of $f$ on
$\GG$ and on $\GO$, respectively, then we are looking at the problem
\begin{equation} \label{abstract max}
\begin{cases}
|f(z)| \to \max
\\
\|f\|_{\PH} \leq 1
\\
\|f\|_\Gamma \leq \epsilon
\end{cases}
\end{equation} 

\noindent Assume that the global norm is induced by an inner product $(\cdot,
\cdot)$ and that the point evaluation functional $f \mapsto f(z)$ is
continuous (for any point $z \in \Omega$), then by the Riesz representation
theorem, there exists an element $p_z \in \PH$ such that $f(z) = (f,p_z)$. Now
inner products with the function $p(\zeta, z):=p_z(\zeta)$ reproduce values of
a function in $\PH$. In this case $\PH$ is called a a reproducing kernel
Hilbert space (RKHS) with kernel $p$. Examples of such spaces include the
Hardy spaces $H^2$ over unit disk, annulus or upper half-plane. From now on we
will drop the subscript $\PH$ for the Hilbert space norm in $\PH$.
\begin{lemma}
  \label{lem:ptaubdd}
Suppose that $\PH$ is a RKHS whose elements are continuous functions on a
metric space $\GO$. Then the function $\GO\ni\tau\mapsto\|p_{\tau}\|$ is
bounded on compact subsets of $\GO$.
\end{lemma}
\begin{proof}
 Assume the contrary. Suppose $S\subset\GO$ is compact, but there exists a
 sequence $\{\tau_k\}_{k=1}^\infty\subset S$, such that $\|p_{\tau_k}\|
 \to \infty$ as $k \to \infty$.  Since $S$ is compact we can extract
 a convergent subsequence (without relabeling it) $\tau_k \to \tau_*$, then
 for any $f \in \PH$ we have $f(\tau_k) = (f, p_{\tau_k}) \to f(\tau_*)=(f,
 p_{\tau_*})$,  by continuity of $f$. Thus, $p_{\tau_k} \weak p_{\tau_*}$ in $\PH$, but this implies boundedness of $\|p_{\tau_k}\|$, leading to a contradiction.
\end{proof}
\begin{corollary}
\label{cor:pbd}
  Under the assumption of Lemma~\ref{lem:ptaubdd} the function $p(\Gz,\tau)$
  is bounded on compact subsets of $\GO\times\GO$, since
$|p(\zeta, \tau)| = |(p_\tau, p_\zeta)|\le\|p_{\tau}\|\|p_{\Gz}\|$.
\end{corollary}

Assume that the smallness on $\Gamma$ is measured in
$L^2:=L^2(\Gamma,|\D\tau|)$-norm (where $|\D \tau|$ is the arc length
measure). Then, 
there is a constant $c>0$ such that
\begin{equation} \label{L^2 norm < H norm first}
\|f\|_{\GG} \leq c \|f\|, \qquad \qquad \forall f \in \PH.
\end{equation} 
Indeed, for all $\tau\in\GG$ we have
$|f(\tau)|=|(f,p_{\tau})|\le\|p_{\tau}\|\|f\|$. Since $\GG$ lies in a compact
subset of $\GO$ and has finite length we conclude by Lemma~\ref{lem:ptaubdd}
that (\ref{L^2 norm < H norm first}) holds.

In order to analyze problem (\ref{abstract max}) we consider a Hermitian
symmetric form 
\[
B:\PH \times \PH \to \mathbb{C},\qquad B(f,g)=(f,g)_{\GG}.
\]
By (\ref{L^2 norm < H norm first}) $B(f,g)$ is continuous, and thus there
exists a positive, self-adjoint and bounded operator $\CK : \PH \to \PH$ with
$B(f,g) = (\CK f, g)$. Moreover we can write an explicit formula for $\CK$ in
terms of the kernel $p$:
\begin{equation} \label{(f,g)=(Kf,g)}
(\CK f, g) =(f, g)_{\GG} = \int_{\Gamma} f(\tau) (p_\tau, g) |\D \tau| = \left( \int_{\Gamma} f(\tau) p_\tau |\D \tau|, g \right).
\end{equation}    
Thus, for every $f\in\PH$
\begin{equation} \label{K RKHS}
(\CK f) (\zeta) = \int_\Gamma  p(\zeta, \tau) f(\tau) |\D \tau|, \qquad \zeta \in \Omega.
\end{equation}
This formula permits to define a new operator $\CK:L^{2}(\GG)\to\PH$. However,
in doing so we may lose injectivity, which underlies uniqueness of analytic
continuation\footnote{It is this property that forces us to restrict attention
to reproducing kernel Hilbert spaces of analytic functions.}. Therefore, we
restrict the domain of $\CK$ to a closed subspace of $L^{2}(\GG)$
\begin{equation} \label{W}
\PW =\text{cl}_{L^2} \left( \PH |_\Gamma \right)\subset
L^{2}(\GG).
\end{equation}
\noindent In fact, in many cases $\PW = L^2(\Gamma)$. The density in the
context of Hardy spaces is known as the Riesz theorem (see
e.g. \cite{part97}). If $\GO$ is bounded it is usually proved using density of
polynomials in $L^{2}(\GG)$, which always holds if all polynomials are in
$\PH$ (and $\Gamma$ is not a closed curve).

We note that the operator $\CK:\PW\to\PH$ is bounded.  Indeed, by
Corollary~\ref{cor:pbd} the function $\GG\ni\tau\mapsto p(\Gz,\tau)$ is
bounded for each $\Gz\in\GO$ and by (\ref{(f,g)=(Kf,g)}) we
have
\begin{equation} \label{K L^2 to H}
\|\CK f\|^2 = (\CK f,\CK f)=(f, \CK f)_{\GG} \leq \|\CK f\|_{\GG} \|f\|_{\GG} \leq c \|\CK f\| \|f\|_{\GG},
\end{equation} 
where we have used (\ref{L^2 norm < H norm first}) in the last inequality. It
follows that $\|\CK f\|\le c\|f\|_{\GG}$.

The outcome of our constructions is the ability to write the two inequalities
in (\ref{abstract max}) as quadratic constraints for $f\in\PH$:
\begin{equation}
  \label{quadconstr}
  (f,f)\le 1,\qquad (\CK f,f)\le\Ge^{2}.
\end{equation}
The final observation is that the objective functional $|f(z)|$ in
(\ref{abstract max}) can be replaced by a (real) linear functional
$\Re(f,p_{z})$. Indeed,
\[
|f(z)|=\sup_{|\Gl|=1}\Re(\Gl f(z))=\sup_{|\Gl|=1}\Re(\Gl f,p_{z}).
\]
It remains to notice that if $f$ satisfies (\ref{quadconstr}) then so does
$\Gl f$ for every $\Gl\in\bb{C}$, $|\Gl|=1$. Thus we arrive at the problem

\begin{equation} \label{maximization problem}
\begin{cases}
\Re (f,p_z) \to \text{max}
\\
(f,f) \leq 1
\\
(\CK f, f) \leq \epsilon^2
\end{cases}
\end{equation}

\begin{lemma} \label{LEM K is compact}
The operator $\CK:\PH \to \PH$ is compact, positive definite and self-adjoint.
\end{lemma} 

\begin{proof}
  \noindent Self-adjointness and positivity of $\CK$ on $\PH$ are immediate
  consequences of \eqref{(f,g)=(Kf,g)}. To prove compactness, let
  $\{f_k\}_{k=1}^\infty \subset \PH$ be a bounded sequence. Extract a weakly
  convergent subsequence (without relabeling it) $f_k \rightharpoonup f$. Then
  for every $\tau\in\GO$ we have $f_k(\tau)=(f_{k},p_{\tau}) \to
  (f,p_{\tau})=f(\tau)$. In addition, for every $\tau\in\GG$ we have
  $|f_{k}(\tau)|=|(f_{k},p_{\tau})|\le\|f_{k}\|\|p_{\tau}\|$. The sequence
  $\|f_{k}\|$ is bounded, since $f_{k}$ is weakly convergent, while
  $\|p_{\tau}\|$ is bounded on $\GG$ by Lemma~\ref{lem:ptaubdd}. Thus,
  $f_k(\tau)$ is uniformly bounded on $\GG$. Then $f_{k}|_\GG\to f|_\GG$ in
  the $L^{2}$ norm. But then by the estimate $\|\CK (f_k-f)\| \le c\|f_k-f\|_{\GG}$ (see
  \eqref{K L^2 to H}) we conclude that $\CK f_{k}\to\CK f$ in $\PH$.

\end{proof}

\begin{theorem} \label{THM main RKHS}
Let $\PH=\pazocal{H}(\Omega)$ be a RKHS of functions analytic in domain
$\Omega$, with kernel $p$ and norm $\|\cdot\|$. Let $\Gamma \Subset \Omega$ be
a rectifiable curve of finite length and $\|\cdot\|_\Gamma$ be the $L^2:= L^2(\Gamma, |\D \tau|)$ norm. Fix a point $z \in \Omega \backslash cl(\Gamma)$ and assume $f \in \PH$ with $\|f\| \leq 1$ and $\|f\|_{\GG} \leq \epsilon$, then

\begin{equation} \label{main bound RKHS}
|f(z)| \leq \frac{3}{2} u_{\epsilon, z}(z)\min\left\{\frac{1}{\|u_{\epsilon, z}\|},
\frac{\Ge}{\|u_{\epsilon, z}\|_{\Gamma}}\right\}
\end{equation}

\noindent where $u_{\epsilon, z}\in\PH$ is the unique solution of

\begin{equation} \label{u RKHS}
\CK u + \epsilon^2 u = p_z.
\end{equation}

\noindent Moreover, \eqref{main bound RKHS} is optimal since it is attained (up to
the factor $3/2$) by

\begin{equation} \label{Mepsz}
  M_{\epsilon, z} (\zeta)= u_{\epsilon, z}(\zeta)\min\left\{\frac{1}{\|u_{\epsilon, z}\|},
\frac{\Ge}{\|u_{\epsilon, z}\|_{\Gamma}}\right\}.
\end{equation}
\end{theorem}

Before we prove this theorem several remarks need to be made.
\begin{enumerate}
\item An obvious thing to do is to set $\Ge=0$ in (\ref{u RKHS}). If
  $p_{z}\in\CK(\PW)$, where $\PW$ is given by \eqref{W}, then $u_{\Ge,z}\to u_{0}=\CK^{-1}p_{z}$, as $\Ge\to 0$. In
  which case the upper bound (\ref{main bound RKHS}) is simply 
\begin{equation} \label{complete stability}
|f(z)| \leq C \epsilon,\qquad C=\frac{3u_{0}(z)}{2\|u_{0}\|_{\GG}}.
\end{equation}
In other words we have numerically stable analytic continuation. Examples
where this happens are mentioned in Remarks~\ref{REM stability circle} and
\ref{REM stability in UHP}. This case will be referred to as the trivial case.
\item The function on the \rhs\ of (\ref{Mepsz}) is obviously in $\PH$ and
  obviously satisfies the constraints in (\ref{abstract max}). Hence, the
  attainability of the bound (\ref{main bound RKHS}) is trivial. Only the
  bound itself requires a proof.
\item The upper bound (\ref{main bound RKHS}) is not an explicit function of
  $\Ge$ and $z$. Its asymptotics as $\Ge\to 0$ depends on fine properties of
  the operator $\CK$. This will be discussed in Section~\ref{SEC K from exp
    decay to power law}. In specific examples in
  Section~\ref{sec:applications} equation (\ref{u RKHS}) is solved explicitly
  and the power law behavior $M_{\Ge,z}(z)\sim \Ge^{\Gg(z)}$ is exhibited.
  
\item The precise asymptotics of the exponential decay of eigenvalues of $\CK$
  is known for certain classes of spaces. For example, assume $\PH$ coincides
  with the Smirnov class $E^2(\Omega)$ \cite{duren}. If the domain $\Omega$ is
  bounded and simply connected and $\Gamma \Subset \Omega$ is a closed Jordan
  rectifiable curve of class $C^{1+\epsilon}$ for $\epsilon>0$, with $\Omega'$
  denoting the domain bounded by it, then the eigenvalues of $\CK$ satisfy the
  asymptotic relation \cite{parfenov}

\begin{equation} \label{parfenov}
\lambda_n(\CK) \sim \rho^{2n+1}, \qquad \qquad \text{as} \ n \to +\infty,
\end{equation}

\noindent where $\rho<1$ is the Riemann invariant, whereby the domain $\Omega \backslash cl(\Omega')$ is conformally equivalent to the annulus $\{\omega \in \mathbb{C} : \rho < |\omega| < 1\}$. 
\end{enumerate}

\noindent The proof of Theorem~\ref{THM main RKHS} in the more general context of RKHS
follows without much change from the proof of the same theorem for the Hardy space
$H^{2}$ of analytic functions in the upper half-plane given in
\cite{grho-gen}. For the sake of completeness we give a short recap of the
argument.

\subsection{Proof of Theorem~\ref{THM main RKHS}} \label{SECT maximizing the error}
We start by analyzing the trivial case.
\begin{lemma} \label{LEM complete stability}
Assume the setting of Theorem~\ref{THM main RKHS}, let $p_z \in \CK(\PW)$, then

\begin{equation*}
|f(z)| \leq c \epsilon.
\end{equation*}

\end{lemma}

\begin{proof}
Let $v \in \PW \subset L^2$ satisfy $\CK v = p_z$, (note that $v$ does not depend on $\epsilon$), then using \eqref{(f,g)=(Kf,g)} we have

\begin{equation*}
f(z) = (f, p_z) = (f, \CK v) = (f,v)_{\GG}.
\end{equation*}

\noindent It remains to use the Cauchy-Schwartz inequality to conclude the desired inequality with $c = \|v\|_{\GG}$.

\end{proof}

\noindent Let us now turn to the case $p_z \notin \CK(\PW)$.
For every $f$, satisfying \eqref{quadconstr} and for every nonnegative numbers $\mu$ and $\nu$ ($\mu^{2}+\nu^{2}\not=0$) we have the inequality
\begin{equation}
  \label{Lagrange}
  ((\mu + \nu \CK)f, f)\le\mu+\nu \epsilon^{2}.
\end{equation}
Applying convex duality to the quadratic functional on the \lhs\ of
(\ref{quadconstr}) we get
\begin{equation}\label{cxdual}
\Re(f, p_z)-\frac{1}{2} \left((\mu+\nu \CK)^{-1} p_z, p_z \right)\le
\frac{1}{2} \left( (\mu+\nu \CK)f,f \right)\le
\frac{1}{2} \left( \mu+\nu \epsilon^{2}  \right),
\end{equation}

\noindent so that

\begin{equation} \label{maxub}
\Re(f, p_z) \le \frac{1}{2} \left((\mu+\nu \CK)^{-1} p_z, p_z \right)
+\frac{1}{2} \left( \mu+\nu \epsilon^{2}  \right),
\end{equation}

\noindent which is valid for every $f$, satisfying \eqref{quadconstr} and all $\mu > 0$, $\nu\ge
0$. In order for the bound to be optimal we must have equality in
(\ref{cxdual}), which holds if and only if
\[
p_z = (\mu + \nu \CK) f,
\]
giving the formula for optimal vector $f$:

\begin{equation} \label{maxxi}
  f=(\mu + \nu \CK)^{-1} p_z.
\end{equation}
The goal is to choose the Lagrange multipliers $\mu$ and $\nu$ so that the
constraints in (\ref{maximization problem}) are satisfied by $f$, given by (\ref{maxxi}).

\vspace{.1in}

\noindent $\bullet$ if $\nu = 0$, then $f = \frac{p_z}{\mu}$ and optimality implies that the first inequality constraint of \eqref{maximization problem} must be attained, i.e. $\|f\|=1$. Thus, $f = \frac{p_z}{\|p_z\|}$ does not depend on the small parameter $\epsilon$, which leads to a contradiction, because the second constraint $(\CK f,f) \leq \epsilon^2$ is violated if $\epsilon$ is small enough. 

\vspace{.1in}

\noindent $\bullet$ if $\mu= 0$, then $\CK f = \frac{1}{\nu}p_z$. But this
equation has no solutions in $\PH$ according to the assumption $p_z \notin \CK(\PW)$.

\vspace*{.1in}

Thus we are looking for $\mu>0,\  \nu> 0$, so that equalities in
\eqref{maximization problem} hold. (These are the complementary slackness relations in Karush-Kuhn-Tucker conditions.), i.e.

\begin{equation}\label{cxdeq}
\begin{cases}
\left( (\mu + \nu \CK)^{-1} p_z, (\mu + \nu \CK)^{-1} p_z \right) =1, \\
\left( \CK (\mu + \nu \CK)^{-1} p_z, (\mu + \nu \CK)^{-1} p_z \right)= \epsilon^2.
\end{cases}
\end{equation}

\noindent Let $\eta = \frac{\mu}{\nu}$, we can solve either the first or the
second equation in \eqref{cxdeq} for $\nu$

\begin{equation} \label{nu1}
  \nu^{2} = \|(\CK+\eta)^{-1} p_z\|^{2},
\end{equation}
or
\begin{equation} \label{nu2}
  \nu^{2}=\epsilon^{-2}\left( \CK (\eta + \CK)^{-1} p_z, (\eta + \CK)^{-1} p_z \right).
\end{equation}

\noindent The two analysis paths stemming from using one or the other representation for $\nu$ lead to two versions of the upper bound on $|f(z)|$, 
optimality of neither we can prove. However, the minimum of the two upper
bounds is still an upper bound and its optimality is then apparent. At first
glance both expressions for $\nu$ should be equivalent and not lead
to different bounds. Indeed, their equivalence can be stated as an equation

\begin{equation}\label{Phi(eta) def and equation}
\Phi(\eta):= \frac{ \left( \CK (\CK + \eta)^{-1} p_z, (\CK + \eta)^{-1} p_z \right) }
{ \|(\CK+\eta)^{-1} p_z\|^2 } = \epsilon^{2}
\end{equation}

\noindent for $\eta$. Equation \eqref{Phi(eta) def and equation} has a unique
solution $\eta_{*}=\eta_{*}(\Ge) > 0$, because $\Phi(\eta)$ is monotone
increasing (since its derivative can be shown to be positive),
$\Phi(+\infty)=(\CK p_{z},p_{z})/\|p_{z}\|^{2}$ and $\Phi(0^{+})=0$. (See
\cite{grho-gen} for technical details.)

In the examples in this paper the eigenvalues and eigenfunctions of $\CK$
exhibit exponential decay. We have shown in \cite{grho-gen} that such behavior
implies that $\eta_{*}(\Ge)\simeq\Ge^{2}$, as $\Ge\to 0$. However, \emph{any}
choice of $\eta$ gives two valid upper bounds: one via (\ref{nu1}), the other,
via (\ref{nu2}). In the anticipation that the exponential decay of eigenvalues
and eigenfunctions holds we simply set $\eta=\Ge^{2}$ and obtain,
setting $u=(\CK+\Ge^{2})^{-1}p_z$,
\begin{equation*}
\Re(f,p_z)\le \frac{(u,p_z)}{2\|u\|}+\Ge^{2}\|u\|,
\qquad \qquad
\Re(f,p_z)\le \frac{\epsilon (u,p_z)}{2\|u\|_{\Gamma}}+\Ge\|u\|_{\Gamma}
\end{equation*}

\noindent Definition of $u$ implies $u(z) = (u, p_z) = (u, \CK u + \epsilon^2 u) = (u, \CK u) + \epsilon^2 (u,u)$, i.e.

\begin{equation} \label{u(z) and the norms}
u(z) = \|u\|_{\GG}^2 + \epsilon^2 \|u\|^2,
\end{equation}

\noindent which implies the inequalities

\[
\Ge^{2}\|u\|\le\frac{u(z)}{\|u\|},\qquad
\|u\|_{L^{2}}\le\frac{u(z)}{\|u\|_{\Gamma}}.
\]
Therefore, we have both
\[
|f(z)|=\Re(f,p_z)\le\frac{3}{2}\frac{u(z)}{\|u\|},
\qquad \qquad
|f(z)|\le\frac{3\Ge}{2}\frac{u(z)}{\|u\|_{\Gamma}}.
\]
Inequality \eqref{main bound RKHS} is now proved. 
We remark that a possibly suboptimal choice $\eta=\Ge^{2}$ still
delivers asymptotically optimal upper bound (\ref{main bound RKHS}), since it
is attained by the function (\ref{Mepsz}).

\subsection{Solving the integral equation} 
\label{SEC K from exp decay to power law}
We begin by making several observations about a priori properties of the
solution $u_{\Ge}$ of (\ref{u RKHS}) in the non-trivial case
$p_{z}\not\in\CK(\PW)$. The most immediate consequence of the non-triviality
is that $\|u_{\Ge}\|_{\GG}$ blows up as $\Ge\to 0$. If it did not, we would be
able to extract a weakly convergent subsequence $u_{\Ge_{k}}\weak u_{0}\in
\PW$ and passing to the weak limits in (\ref{u RKHS}) obtained that $(\CK
u_{0})(\Gz)=p_{z}(\Gz)$, for $\Gz\in\GG$. However, since
$\CK(\PW)\subset\PH$ we get a contradiction with the non-triviality.

Next let us show that equation (\ref{u(z) and the norms}) implies that
$M_{\Ge,z}(z)\gg\Ge$. On the one hand, dividing equation (\ref{u(z) and the norms}) by
$\|u_{\Ge}\|_{\GG}$ we obtain
\[
\frac{u_{\Ge}(z)}{\|u_{\Ge}\|_{\GG}}\ge\|u_{\Ge}\|_{\GG}.
\]
On the other, we have $\|u_{\Ge}\|_{\GG}^2 + \epsilon^2
\|u_{\Ge}\|^2\ge2\epsilon\|u_{\Ge}\|_{\GG} \|u_{\Ge}\|$ and therefore
\[
\frac{u_{\Ge}(z)}{\Ge\|u_{\Ge}\|}\ge2\|u_{\Ge}\|_{\GG},
\]
proving that $\Ge^{-1}M_{\Ge,z}(z)\ge\|u_{\Ge}\|_{\GG}\to+\infty$. This means
that one cannot expect full numerical stability of analytic continuation. 

Finally, we prove the ``mathematical well-posedness'' of analytic
continuation: $M_{\Ge,z}(z)\to 0$ as $\Ge\to 0$.  This is a consequence of the
weak convergence of $u_{\Ge}/\|u_{\Ge}\|$ to 0. If we divide (\ref{u RKHS}) by
$\|u_{\Ge}\|$ and pass to weak limits, using the fact that $\|u_{\Ge}\|\ge
c^{-1}\|u_{\Ge}\|_{\GG}\to+\infty$ we obtain that the weak limit $\Hat{u}$ of
$u_{\Ge}/\|u_{\Ge}\|$ satisfies $\CK\Hat{u}=0$.  But if $\CK \Hat{u}=0$, then
$\|\Hat{u}\|_{\GG}^{2}=(\CK \Hat{u},\Hat{u})=0$. It follows that the analytic
function $\Hat{u}=0$ on $\GG$ and hence must vanish everywhere in $\GO$. This
shows that the operator $\CK$ has a trivial null-space and that
$M_{\Ge,z}(z)=(u_{\Ge}/\|u_{\Ge}\|,p_{z})\to 0$, as $\Ge\to 0$.

A consequence of the just established strict positivity of $\CK$ is
separability of the Hilbert space $\PH$. This should not be surprising, since
$\PH$ consists of analytic functions each of which can be completely described by
a countable set of numbers.
\begin{lemma}
The Hilbert space $\PH$ is always separable.
\end{lemma}
\begin{proof}
  We saw that $\CK : \PH \to \PH$ given by \eqref{K RKHS} is a self-adjoint,
  compact operator. We have just seen that $\CK$ has a trivial null-space.  In
  this case the Hilbert space $\PH$ is the orthogonal sum of countable number
  of finite dimensional eigenspaces of $\CK$ with positive eigenvalues. Thus,
  $\PH$ has a countable complete orthonormal set and is therefore separable.
\end{proof}

In applications of our theory in Section~\ref{sec:applications} we solve
equation \eqref{u RKHS} exactly by finding all eigenvalues and
eigenfunctions of $\CK$.  Let $\{e_n\}_{n =1}^\infty$ be an orthonormal
eigenbasis of $\PH$ with $\CK e_n = \lambda_n e_n$. In this basis the equation
\eqref{u RKHS} diagonalizes:
\begin{equation*}
\Gl_{n} (u,e_{n}) + \Ge^{2} (u,e_{n}) = (p_z,e_{n}),
\end{equation*}
therefore we find
\begin{equation} \label{u in basis}
  u_{\Ge}(\zeta) = \sum_{n} \frac{\overline{e_n(z)}}{\Gl_{n}+\Ge^{2}} e_n(\zeta).
\end{equation}
Using this expansion, formula $\|u\|_{\GG}^2 = (\CK u, u)$, and
(\ref{u(z) and the norms}) we find that
\begin{equation} \label{u(z) and norms of u in series}
u_{\Ge}(z) = \sum_{n} \frac{|e_n(z)|^2}{\lambda_n + \epsilon^2},
\qquad
\|u_{\Ge}\|^2 = \sum_{n} \frac{|e_n(z)|^2}{(\lambda_n + \epsilon^2)^2},
\qquad
\|u_{\Ge}\|_{\GG}^2 = \sum_{n} \frac{\lambda_n |e_n(z)|^2}{(\lambda_n + \epsilon^2)^2}.
\end{equation}
It follows that
\begin{equation} \label{a}
\sum_{n} \frac{|e_n(z)|^2}{\lambda_n} = \infty,
\end{equation}
since if the series had a finite sum then formula (\ref{u(z) and norms
    of u in series}) for $\|u_{\Ge}\|_{\GG}$ would imply
\[
\|u_{\Ge}\|_{\GG}^2\le\sum_{n} \frac{|e_n(z)|^2}{\lambda_n},
\]
contradicting the blow up of $\|u_{\Ge}\|_{\GG}$. 

In our examples where the eigenvalues $\Gl_{n}$ and eigenfunctions
$e_{n}(\Gz)$ can be found explicitly they are seen to decay exponentially fast
to 0 (see also \eqref{parfenov}). As we have shown in \cite{grho-gen} this implies the power law principle
\begin{equation} \label{powerlaw}
 M_{\epsilon, z} (z) \simeq  \epsilon^{\gamma(z)}, \qquad \text{as} \qquad \epsilon \to 0,
\end{equation}
where $\gamma(z) \in (0,1)$ can be expressed in terms of the rates of
exponential decay of spectral data for $\CK$.

\begin{theorem} \label{connections THM}
Let $\{e_n\}_{n=1}^\infty$ be orthonormal eigenbasis of $\PH$ with $\CK e_n = \lambda_n e_n$. Let $u = u_{\epsilon,z}$ and $M_{\epsilon,z}$ be given by \eqref{u RKHS} and \eqref{Mepsz} respectively. Assume

\begin{equation} \label{lambda_n e_n exp decay}
\lambda_n \simeq e^{-\alpha n}, \qquad |e_n(z)|^2 \simeq e^{-\beta n},\qquad
0<\Gb<\Ga,
\end{equation}

\noindent with implicit constants independent of $n$ (so that \eqref{a} holds). Then,

\begin{equation*}
\|u_{\epsilon,z}\|_{\GG} \simeq \epsilon \|u_{\epsilon,z}\| \simeq \epsilon^{\frac{\beta}{\alpha} - 1}
\qquad \text{and} \qquad
u_{\epsilon,z} (z) \simeq \epsilon^{2\left(\frac{\beta}{\alpha} - 1\right)},
\end{equation*}

\noindent with implicit constants independent of $\epsilon$. In particular,
this implies the power law principle (\ref{powerlaw}) with exact exponent:
\begin{equation*}
  M_{\epsilon,z}(z) \simeq \epsilon^{\frac{\beta}{\alpha}}.
\end{equation*}

\end{theorem}

\vspace{.1in}

\noindent The proof of Theorem~\ref{connections THM} immediately follows from \eqref{u(z) and norms of u in series} and Lemma~\ref{LEM sum asymptotics}.

\subsection{Linear constraints}
\label{sub:lincon}
In one of our examples we encounter a situation where additional linear
constraints are imposed on a previously solved problem. In general all linear
constraints on analytic functions will simply be incorporated into the
definition of the RKHS $\PH$. The question is whether we can use the already
found solution of a problem if additional linear constraints are imposed.  Let
$L \subset \PH$ be a closed, $\mathbb{C}$-linear
subspace. Then $L$ with the inner product from $\PH$ is still a RKHS with the
reproducing kernel $\PP_L p_z$, where $\PP_L$ denotes the orthogonal
projection onto $L$. If we restrict $f$ and $g$ in (\ref{(f,g)=(Kf,g)}) to
elements from $L$, then the operator $\CK$ can be written as $\PP_L \CK
\PP_L$. Then equation (\ref{u RKHS}) can be written (in the language of the
original RKHS $\PH$) as
\begin{equation} \label{u RKHS in L}
\PP_L \CK \PP_L u + \epsilon^2 u = \PP_L p_z,  \qquad u\in \PH,
\end{equation}
whose unique solution $u$ necessarily belongs to $L$.
In general, one's ability to solve the original problem \eqref{u RKHS} would
be of little help for solving \eqref{u RKHS in L}, except in the special case
when $L$ is an invariant subspace of $\CK$. In this case $\PP_L$ commutes with
$\CK$ and if $u$ solves \eqref{u RKHS}, then $\PP_L u$ solves \eqref{u RKHS in
  L}.

The requirement that $L$ be a $\mathbb{C}$-linear subspace is important,
because the linearization argument taking the objective functional $|f(z)|$ in
(\ref{abstract max}) to the one in (\ref{maximization problem})
requires all the constraints to be invariant under
multiplication by a phase factor $\Gl\in\bb{C}$, $|\Gl|=1$. In some
applications, like the analytic continuation of the complex electromagnetic
permittivity the constraints may be just $\RR$-linear, in which case other
techniques have to be applied \cite{grho-CEMP}.
 
\section{Applications}
\setcounter{equation}{0}
\label{sec:applications}
\subsection{The annulus} \label{SECT annulus proof}

Here we prove Theorem~\ref{THM annulus}, so assume the setting of
Section~\ref{SECT annulus}. Note that if we replace $H^2$-norm in
Theorem~\ref{THM annulus} by another equivalent norm, this will only change
the constant $C$ in the inequality \eqref{main bound annulus}. In order to
apply our theory we need a norm, induced by an inner product, with respect to
which the reproducing kernel of the space $H^2$ is as simple as possible. To
define such an inner product we use the Laurent expansion 
\begin{equation} \label{f+- def}
f(\zeta) = \sum_{n \geq 0} f_n \zeta^n + \sum_{n < 0} f_n \zeta^n =: f_+(\zeta) + f_-(\zeta),
\end{equation}

\noindent then $f \in H^2(A_\rho)$ \IFF $f_+ \in H^2(\{|\zeta|<1\})$ and $f_- \in H^2(\{|\zeta| > \rho\})$ (cf. \cite{sara65}). So we define

\begin{equation} \label{inner product H^2}
(f,g) = \tfrac{1}{2\pi} (f_+, g_+)_{L^2(\Gamma_1)} + \tfrac{1}{2\pi \rho} (f_-, g_-)_{L^2(\Gamma_\rho)},
\end{equation}

\noindent The norm in $H^2(A_\rho)$ induced by \eqref{inner product
  H^2} is equivalent to the norm \eqref{H^2 annulus def} (e.g. \cite{sara65,koos98}). Now the functions $\{\zeta^n\}_{n\in\ZZ}$ form a basis in $H^2(A_\rho)$, let us normalize them:

\begin{equation} \label{e_n annulus}
e_n(\zeta)=
\begin{cases}
\zeta^n, &\qquad n \geq 0
\\
(\zeta / \rho)^n, &\qquad n<0,
\end{cases}
\end{equation}

\noindent then $\{e_n\}_{n \in \ZZ}$ is orthonormal basis of $H^2(A_\rho)$. Definition of the reproducing kernel implies that $p(\zeta, \tau) = \sum_n \overline{e_n(\tau)} e_n(\zeta)$. Computing this sum, or by adding kernels of the spaces $H^2(\{|\zeta|<1\})$ and $H^2(\{|\zeta|>\rho\})$, we find the reproducing kernel of $H^2(A_\rho)$:

\begin{equation} \label{p annulus}
p(\zeta, \tau) = \frac{1}{1-\zeta \overline{\tau}} + \frac{\rho^2}{\zeta \overline{\tau} - \rho^2}.
\end{equation}

\noindent Note that $p_z \notin \CK(\PW)$. Indeed, the function $p_{z}$ has
simple poles at $\overline{z}^{-1}, \rho^2 \overline{z}^{-1}$. At the same time, for
any $f \in \PW \subset L^2(\Gamma)$ the function $\CK f$ may have singularities only in
the set $S=\cup_{\tau \in \Gamma} \{\overline{\tau}^{-1}, \rho^2
\overline{\tau}^{-1}\}$. If $\overline{z}^{-1}\in S$, then
$z\in\GG\cup\rho^{-2}\GG$. If $\rho^2 \overline{z}^{-1}\in S$, then
$z\in\GG\cup\rho^{2}\GG$. 
But since $z \notin \Gamma$ and curves $\rho^{\pm2}\GG$ are outside of the
annulus $A_{\rho}$, the equation $\CK f (\zeta) = p(\zeta,z)$ for $\zeta \in
A_\rho$ cannot have any solutions in $\PW$.  

We observe that for any orthonormal basis $\{e_{n}:n\in\bb{Z}\}$ of $\PH$ we have, using \eqref{(f,g)=(Kf,g)},
\begin{equation} \label{K annulus}
\CK f = \sum_{n \in \ZZ}(\CK f,e_{n})e_{n}=\sum_{n \in \ZZ} (f,e_n)_{L^2(\Gamma)} \ e_n.
\end{equation}
It is easy to verify that when $\Gamma$ is a circle centered at the origin,
the functions $\{e_n\}$, given by (\ref{e_n annulus}) are also orthogonal
in $L^2(\Gamma)$ and hence, taking $f=e_{m}$ in (\ref{K annulus}) we
conclude that $\CK e_m = \|e_m\|_{L^2(\Gamma)}^2 e_m$. So we have proved
\begin{lemma}
Let $\{e_n\}_{n \in \ZZ}$ be given by \eqref{e_n annulus} and $\CK$ given by \eqref{K annulus}, then

\begin{equation*}
\CK e_n = \lambda_n e_n, \qquad \qquad n \in \ZZ,
\end{equation*}

\noindent where

\begin{equation} \label{lambda_n annulus}
\lambda_n = 2\pi r
\begin{cases}
r^{2n}, &\qquad n \geq 0
\\
(r / \rho)^{2n}, &\qquad n<0
\end{cases}
\end{equation}
 
\end{lemma}

We see that $\lambda_n$ and $|e_n(z)|$ approach to zero along two different sequences and have two different asymptotic behaviors, which are distinguished by the location of $z$ relative to $\Gamma$. Therefore, to apply Theorem~\ref{connections THM} we need to consider two cases. Assume that $z$ lies outside of $\Gamma$, i.e. $|z| \in (r,1)$. The function $u$ from \eqref{u RKHS} is given by

\begin{equation} \label{u annulus}
u(\zeta) = \sum_{n \in \ZZ} \frac{\overline{e_n(z)} e_n(\zeta)}{\lambda_n + \epsilon^2}.
\end{equation}

\noindent Note that, for any $n \in \ZZ$

$$\frac{|e_n(z)|^2}{\lambda_n} = \frac{1}{2\pi r} \left( \frac{|z|}{r} \right)^{2n}.$$

\noindent By assumption the above quantity is summable over $n<0$, this
implies that in analyzing $u(z)$ the sum over negative indices is $O(1)$, as
$\epsilon\to 0$, and hence can be ignored. The dominant part is the sum over $n \geq 0$. Analogously, in quantities $\|u\|_{H^2(A_\rho)}, \|u\|_{L^2(\Gamma)}$ as well, the sum can be restricted to $n \geq 0$. This determines the behaviors $\lambda_n \simeq r^{2n}$ and $|e_n(z)| \simeq |z|^n$, therefore Theorem~\ref{connections THM} implies that the exponent is $\gamma(z) = \frac{\ln |z|}{\ln r}$. The case $|z| \in (\rho, r)$ is done analogously and \eqref{gamma annulus} now follows.

Next, we can rewrite \eqref{u annulus} as 

\begin{equation}
u(\zeta) = \sum_{n \geq 0} \frac{\overline{z}^n \zeta^n}{2\pi r r^{2n} + \epsilon^2} + \sum_{n <0} \frac{\overline{z}^n \zeta^n}{2\pi r r^{2n} + \epsilon^2 \rho^{2n}}.
\end{equation}

\noindent Let us consider the function

\begin{equation}
\tilde{u}(\zeta) = \sum_{n \in \ZZ} \frac{\overline{z}^n \zeta^n}{r^{2n} + \epsilon^2(1+\rho^{2n})},
\end{equation}

\noindent clearly for negative indices $\rho^{2n} \ll 1$ and hence can be ignored, and for positive indices $1$ can be ignored from the denominator in the definition of $\tilde{u}$. Therefore, values of $\tilde{u}, u$ at $z$ and their $H^2$ and $L^2$-norms have the same behavior in $\epsilon$. Thus, we may consider $\tilde{u}$ instead, which then gives rise to the maximizer function $M$ in \eqref{M annulus}. Finally, the fact that $\|M\|_{H^\infty (\overline{A_\rho})}$ is bounded uniformly in $\epsilon$ follows from the application of Lemma~\ref{LEM sum asymptotics}.

\subsection{The upper half-plane}
\setcounter{equation}{0} 
\label{sec:uhp}

\noindent \textbf{Notation:} Let $D(c,r)$ and $C(c,r)$ denote respectively the closed disk and the circle centered at $c$ and of radius $r$ in the complex plane.

\vspace{.1in} In this section we prove Theorem~\ref{THM H+}. The Hardy space
$H^2(\HH_+)$ of functions analytic in the complex upper half-plane $\HH_+$ is
a RKHS with the inner product $(f,g) = (f,g)_{L^2(\RR)}$. By Cauchy's integral formula
\begin{equation*}
f(z) = \frac{1}{2\pi i} \int_{\RR} \frac{f(x) \D x}{x-z}.
\end{equation*}

\noindent Therefore, the reproducing kernel $p$ of $H^2(\HH_+)$ is
\begin{equation*}
p_{\tau}(\Gz)=p(\zeta, \tau) = \frac{i}{2\pi (\zeta - \overline{\tau})},\qquad \{\Gz,\tau\}\subset\HH_+.
\end{equation*}

\noindent In Theorem~\ref{THM H+} the data is measured on $\Gamma = C(i,r)$ with $r \in (0,1)$. Using the definition of $\CK$ \eqref{K RKHS} we have

\begin{equation*}
\CK u (\zeta) = \frac{1}{2\pi} \int_{\Gamma} \frac{i u(\tau) |\D \tau|}{\zeta - \overline{\tau}}.
\end{equation*}

\noindent Note that $p_z \notin \CK(\PW)$. Indeed, the function $p_{z}$ is
analytic everywhere in $\bb{C}$, except at $\bra{z}$, where it has a
pole. At the same time for any $f \in \PW \subset L^2(\Gamma)$ the function $\CK f$
is analytic everywhere in $\bb{C}$ outside of $\overline{\Gamma}$. But
$\overline{z}\notin\overline{\Gamma}$, since $z$ lies outside of
$\Gamma$. Therefore, the equation $\CK f = p_{z}$ has no solutions in $\PW$.

\begin{lemma}
\label{lem:eigenstuff}
Let $r \in (0,1)$ and $\Gamma = C(i,r)$. Let $\{e_n\}_{n=1}^\infty$ be an
orthonormal eigenbasis of $\CK$ in $H^2(\HH_+)$, with eigenvalues $\{\lambda_n\}_{n=1}^\infty$. Then
\begin{equation} \label{lambda_n and e_n of UHP}
\lambda_n = \frac{r \rho^{2n}}{1+\sqrt{1-r^2}}, \qquad \qquad
e_n(\zeta) = \frac{\sqrt[4]{1-r^2}}{\sqrt{\pi}} \frac{m(\zeta)^n}{\zeta + z_0},
\end{equation}

\noindent where $\rho, z_0, m(\zeta)$ are as in Theorem~\ref{THM H+}.
\end{lemma} 

\vspace{.1in}

Before proving this lemma, let us see that it concludes the proof of Theorem~\ref{THM H+} upon the application of Theorems~\ref{THM main RKHS} and \ref{connections THM}. Indeed, $\lambda_n \simeq \rho^{2n}$ and $|e_n(z)| \simeq |m(z)|^n$, then the formula \eqref{gammaUHP} for the exponent $\gamma(z)$ follows. The function $u$ from \eqref{u RKHS} is given by

\begin{equation*}
u(\zeta) = \frac{\pi^{-1} \sqrt{1-r^2}}{(\overline{z} + \overline{z}_0)(\zeta + z_0)} \sum_{n=1}^\infty \frac{\overline{m(z)}^n m(\zeta)^n}{\frac{r}{1+\sqrt{1-r^2}} \rho^{2n} +\epsilon^2 }.
\end{equation*}

\noindent As in the case of annulus, ignoring the constants that don't affect
the asymptotics of the function as $\epsilon\to 0$ we obtain the maximizer \eqref{M
  upper halfplane}. 

\begin{proof}[Proof of Lemma~\ref{lem:eigenstuff}]
Let $\CK w(\zeta) = \lambda w(\zeta)$, then $w$ must be analytic in the
extended complex plane with the closed disk $D(-i,r)$ removed. In particular,
it is analytic in $D(i,r)$. Thus, we can evaluate the
operator $\CK$ explicitly in terms of values of $w$.
\[
\CK w(\zeta) = \frac{1}{2\pi} \int_{0}^{2\pi}\frac{irw(i+re^{it})dt}{\zeta+i-re^{-it}} = \frac{1}{2\pi}
\int_{C(0,r)}\frac{r w(i+\tau) d\tau}{(\zeta+i)\tau-r^{2}}.
\]
We note that $r^{2}/|\zeta+i| <r$ precisely when $\zeta$ is outside of the closed disk
$D(-i,r)$. In addition $w(i+\tau)$ is analytic in $D(0,r)$, hence
\[
\CK w(\zeta) = \frac{i r}{\zeta+i} w \left( i + \frac{r^{2}}{\zeta+i} \right).
\]
Next we note that the M\"obius transformation 
\[
\sigma(\zeta)=i+\frac{r^{2}}{\zeta+i}
\]
maps $D(-i,r)$ onto the exterior of $D(i,r)$. In particular there is a disk
$D_{1}\subset D(-i,r)$ such that $\sigma(D_{1})=D(-i,r)$. Then $\CK w$ is
analytic in the exterior of $D_{1}$, since $w$ is analytic outside of
$D(-i,r)$. But $w$ is an eigenfunction of $\CK$, hence it must also be
analytic outside of $D_{1}$. Repeating the argument using the fact that $w$ is
analytic in the larger domain $\bb{C}\setminus D_{1}$ we conclude that it must
also be analytic outside of $D_{2}\subset D_{1}$, such that
$\sigma(D_{2})=D_{1}$. We can continue like this indefinitely, showing that
the only possible singularity of $w$ must be at the fixed point $\zeta_{0}\in
D(-i,r)$ of $\sigma(\zeta)$. We find
\[
\zeta_{0}=-i\sqrt{1-r^{2}}.
\]
Since $w$ is analytic at infinity the transformation $\eta=1/(\zeta-\zeta_{0})$ will map the extended complex plane with $\zeta_{0}$ removed to the entire complex plane (without the infinity). The eigenfunction $w$ will then be an
entire function in the $\eta$-plane.
Let $v(\eta)=w(\eta^{-1}+\zeta_{0})$. Then
\[
w(\zeta) = v \left(\nth{\zeta-\zeta_{0}}\right).
\]
The relation $\CK w = \lambda w$ now reads
\[
\Gl v(\eta) = \frac{i r \eta}{\eta (\zeta_{0}+i)+1} v \left(\frac{\eta (\zeta_{0}+i)+1}{i-\zeta_{0}}\right).
\]
One corollary of this equation is that $v(0)=0$. Hence,
$\phi(\eta)=\eta^{-1}v(\eta)$ is also an entire function, satisfying
\[
\Gl\phi(\eta)=\frac{ir}{i-\zeta_{0}}\phi\left(\frac{\eta(\zeta_{0}+i)+1}{i-\zeta_{0}}\right).
\]
We see that $\phi(\eta)$ is an entire function with the property that $\phi(a
\eta + b)$ is a constant multiple of $\phi(\eta)$, with $b =
\frac{1}{i-\zeta_0}$ and $a=\rho^2$, where $\rho$ is given by
\eqref{gammaUHP}. It remains to observe that such a property holds for functions
$\phi_{n}(\eta)=(\eta-\eta_{0})^{n}$, provided
\[
\frac{\eta_{0}-b}{a} = \eta_{0} \quad \eqv \quad \eta_{0}=\frac{b}{1-a}.
\]
Indeed,
\[
(a \eta + b - \eta_{0})^{n} = a^{n} \left( \eta - \frac{\eta_{0}-b}{a}\right)^{n}=a^{n}(\eta-\eta_{0})^{n}.
\]
In our case we get $\eta_{0}=-\frac{1}{2\zeta_{0}}$ and conclude that $\phi_{n}(\eta)=\left(\eta + \frac{1}{2 \zeta_{0}}\right)^{n}$ and $\lambda_n$ is given by \eqref{lambda_n and e_n of UHP}. Converting the formula back to $w_{n}(\zeta)$ we obtain (up to a constant multiple)
\[
w_{n}(\zeta)=\frac{1}{\zeta - \zeta_{0}}\left(\frac{\zeta + \zeta_{0}}{\zeta - \zeta_{0}}\right)^{n} = \frac{m(\zeta)^n}{\zeta - \zeta_{0}}.
\]

It remains to normalize the eigenfunctions $w_n$. For that we compute
\[
\|w_{n}\|_{H^2(\HH_+)}^{2} = \int_{\RR} |w_n|^2 \D x = \int_{\RR} \frac{\D x}{|x-\zeta_0|^2} = \frac{\pi}{ \sqrt{1-r^2}}.
\]
\end{proof}

\subsection{The Bernstein ellipse} \label{SECT ellipse}

\subsubsection{From the ellipse to the annulus}
\label{sss:el2an}
The ellipse $E_{R}$ is conformally equivalent to a disk or the upper
half-plane. The conformal mapping effecting the equivalence can be written
explicitly in terms of the Weierstrass $\Gz$-function, but the image of the
interval $[-1,1]$ will then be a curve that would not permit any kind of
explicit solution of the resulting integral equation. Instead we use a much
simpler Joukowski function $J(\omega) = \frac{\omega + \omega^{-1}}{2}$ that
will convert the problem in the ellipse to the problem in an annulus with
$\GG$ being a concentric circle inside the annulus. We observe that $J(\Go)$
maps the annulus $\{R^{-1} < |\omega| < R\}$ onto the Bernstein ellipse $E_R$
in 2-1 fashion, meaning that each point in $E_R$ has exactly two (if we count
the multiplicity) preimages in the annulus (note that $J(\omega) =
J(\omega^{-1})$). Moreover, the unit circle gets mapped onto $[-1,1] \subset
E_R$ under $J$. So given a function $F\in H^{\infty}(E_R)$, the function
$f(\zeta):= F(J(R\zeta))$ is analytic in $A_{\rho}$ defined in \eqref{A_rho
  and gamma}, with $\rho = R^{-2}$, has the same $H^{\infty}$ norm, and
satisfies the symmetry property
\begin{equation}
  \label{symmetry}
  f(\overline{\zeta}) = f(\zeta) \qquad \forall |\zeta| = r=\nth{R}.
\end{equation}
Conversely, any function $f\in H^{\infty}(A_{\rho})$, satisfying
(\ref{symmetry}) defines an analytic function in a Bernstein ellipse (with the
same $H^{\infty}$ norm). This is so because (\ref{symmetry}) can also be
written as 
\begin{equation}
  \label{Rsymmetry}
f\left(\nth{R^{2}\Gz}\right)=f(\Gz) \qquad \forall |\zeta| =r.
\end{equation}
The Schwarz reflection principle then guarantees that (\ref{Rsymmetry}) holds
for all $\Gz\in A_{\rho}$. This implies that
$F(u)=f(R^{-1}J^{-1}(u))$ gives the same value for each of the two branches of
$J^{-1}$ and hence defines an analytic function in $E_{R}$. 
Thus, the analytic continuation problem in ellipse reduces to the one in
the annulus, but with an additional symmetry constraint (\ref{symmetry}). 

\subsubsection{The annulus with symmetry}
Let us now define
\begin{equation}
  \label{Hsym}
  \PH=\{f\in H^2(A_\rho):f(\overline{\zeta}) = f(\zeta) \quad \forall |\zeta| = \sqrt{\rho}\}.
\end{equation}
The curve $\GG$ will be a circle $\GG_{r}$ centered at the origin of radius
$r=\sqrt{\rho}$. 

\begin{lemma}[Annulus with symmetry] \label{LEM annulus symm}
Let $0<\rho<1$ and let $z\in\bb{C}$ be such that
$r<|z|<1$. Then there exists $C>0$, such that for every $\Ge>0$ and every 
$f \in \PH$ with $\|f\|_{H^2} \leq 1$ and
$\|f\|_{L^2(\Gamma_r)} \leq \epsilon$ we have the bound
\begin{equation} \label{main bound annulus sym}
|f(z)| \leq C \epsilon^{\gamma(z)},
\end{equation}
where the exponent $\gamma(z)$ is the same as in Theorem~\ref{THM annulus},
i.e. 
\begin{equation}
  \label{gmz}
  \gamma(z) = \frac{\ln|z|}{\ln r}. 
\end{equation}
Moreover, \eqref{main bound annulus sym} is asymptotically optimal as
$\epsilon\to 0$ and the function attaining the bound is  
\begin{equation} \label{M annulus sym}
M(\zeta) = \epsilon^{2 - \gamma(z)} \sum_{n =1}^\infty \frac{\overline{z}^n + (\rho / \overline{z})^{n}}{\rho^{n} + \epsilon^2} [\zeta^n + (\rho/\zeta)^n], \qquad \Gz \in A_\rho.
\end{equation}
\end{lemma}

\begin{proof}
We note that the maximization problem in Lemma~\ref{LEM annulus symm} differs
from the one in Theorem~\ref{THM annulus} by the requirement of symmetry
(\ref{symmetry}). Hence, following the theory in Section~\ref{sub:lincon} we define
the subspace
\begin{equation*}
L = \{f \in H^2(A_\rho): f(\zeta) = f (\overline{\zeta}) \qquad \forall \zeta
\in \Gamma_r\},\quad r=\sqrt{\rho}.
\end{equation*}
Then, the orthogonal projection onto $L$ will be given by
\begin{equation} \label{P_L in annulus}
\PP_L f(\zeta) = \frac{f(\zeta) + f (\rho / \zeta)}{2}.
\end{equation}
\begin{lemma}
  \label{lem:Linv}
  The integral operator $\CK$ with kernel (\ref{p annulus}) and $\GG=\GG_{r}$
  commutes with $\PP_{L}$.
\end{lemma}
\begin{proof}
  \noindent The commutation $\PP_L \CK = \CK \PP_L$ is then equivalent to

\begin{equation*}
\int_{\Gamma_r} p(\zeta, \tau) u(r^2 / \tau) |\D \tau| = \int_{\Gamma_r} p(r^2 / \zeta, \tau) u(\tau) |\D \tau|
\end{equation*}

\noindent which, after change of variables on the left-hand side reduces to

$$p(\zeta, \rho/\tau) = p(\rho / \zeta, \tau) \qquad \qquad \forall \zeta \in A_\rho, \ \forall \tau \in \Gamma_r.$$ 
Substituting the definition of $p$ from \eqref{p annulus} into this formula we
easily verify it.
\end{proof}
According to the theory in Section~\ref{sub:lincon} the solution of (\ref{u
  RKHS in L}) is $u_{L}=\PP_L u$, where $u$ is given by \eqref{u annulus}.
We observe that in the case $r^2 = \rho$ we have $\lambda_n = \lambda_{-n}$ and $e_n(\rho/\zeta) = e_{-n}(\zeta)$, so that
\begin{equation}
u_{L}=\PP_L u(\zeta) = \frac{1}{1+\epsilon^2} + \frac{1}{2} \sum_{n=1}^\infty \frac{\overline{e_n(z)} + \overline{e_{-n}(z)}}{\lambda_n + \epsilon^2} [e_n(\zeta)+e_{-n}(\zeta)]
\end{equation}
Substituting the expressions for $\lambda_n, e_n$ from (\ref{lambda_n
  annulus}), (\ref{e_n annulus}), respectively, and ignoring the first $O(1)$ term and some constants, which affect the asymptotics of $u_L$ by constant factors, we arrive at the function
\begin{equation*}
u_L(\zeta) = \sum_{n=1}^\infty \frac{\overline{z}^n + (\rho / \overline{z})^n}{\rho^n + \epsilon^2} [\zeta^n + (\rho / \zeta)^n].
\end{equation*}
We note that
\[
e_{n}^{L}=\frac{1}{2}\left(\zeta^n + (\rho / \zeta)^n\right),\quad n\ge 0,
\]
is the orthonormal eigenbasis of $L$ with respect to $\PP_{L}\CK\PP_{L}$. The
corresponding eigenvalues are $\Gl_{n}=2\pi\sqrt{\rho} \rho^{n}$, and for $|z|
\in (r,1)$ we have $|e_{n}^{L}(z)|\simeq |\overline{z}^n + (\rho / \overline{z})^n| \simeq |z|^n$. 
Then, Theorem~\ref{connections THM} gives formula (\ref{gmz}) as well as the maximizer function
(\ref{M annulus sym}).
\end{proof}

\subsubsection{From the annulus to the ellipse}
In this section we will show that Theorem~\ref{THM ellipse} follows from
Lemma~\ref{LEM annulus symm}. Let $F\in H^{\infty}(E_{R})$ be such that
$\|F\|_{H^{\infty}}\le 1$ and $|F(x)|\le\Ge$ for all $x\in[-1,1]$.  As
discussed in Section~\ref{sss:el2an}, the function $f(\zeta):= F(J(R\zeta))$
is analytic in $A_{\rho}$, with $\rho = R^{-2}$ and has the symmetry
$f(\overline{\zeta}) = f(\zeta) \quad \forall |\zeta| = r,$ \ where $r =
R^{-1}$. It also satisfies
$$\|f\|_{H^2(A_\rho)} \lesssim \|F\|_{H^\infty(E_R)} \leq 1$$
as well as
\begin{equation*}
\|f\|_{L^2(\Gamma_r)}^2 = \frac{1}{R} \int_{0}^{2\pi} |F(J(e^{it}))|^2 \D t \le\frac{2\pi\Ge^{2}}{R}.
\end{equation*}
Let $z\in E_{R}\setminus[-1,1]$. Let $z_{a}\in A_{\rho}$ be the unique
solution of $J(Rz_{a})=z$, satisfying $|z_{a}|>r$. Then
by Lemma~\ref{LEM annulus symm} (with $\rho=R^{-2}$ and $r = R^{-1}$) we have
\begin{equation*}
|F(z)|=|f(z_a)| \leq C \epsilon^{-\frac{\ln |z_a|}{\ln R}}=C \epsilon^{1-\frac{\ln\left|J^{-1}(z)\right|}{\ln R}}=C \epsilon^{\alpha(z)},
\end{equation*}
where $\Ga(z)$ is given by (\ref{alpha}). This proves (\ref{main bound ellipse}).

In order to prove the optimality of the bound (\ref{main bound ellipse}) we use
Lemma~\ref{LEM sum asymptotics} to show that $M(\Gz)$ given by (\ref{M
  annulus sym}) satisfies
\[
\begin{cases}
  |M(\Gz)|\lesssim \epsilon,&|\Gz|=r,\\
  |M(\Gz)|\lesssim 1, & r<|\Gz|<1.
\end{cases}
\]
Using the Joukowski function to
map this to a function on the Bernstein ellipse we obtain
\begin{equation}
  \label{ChebM}
  M_{\rm ellipse}(\Go)=M\left(R^{-1}J^{-1}(\Go)\right)=\Ge^{2-\Ga(z)}\sum_{n=1}^{\infty}
\frac{\bra{T_{n}(z)}T_{n}(\Go)}{1+\Ge^{2}R^{2n}},
\end{equation}
where $T_n$ is the Chebyshev polynomial of degree $n$. Chebyshev
polynomials are just monomials $\Gz^{n}$ in the annulus after the Joukowski transformation:
\begin{equation*}
J^{-1} \circ T_n \circ J = \Gz\mapsto\zeta^n, \qquad \forall \zeta \neq 0.
\end{equation*} 
We note that due to the choice of the branch of $J^{-1}$ to correspond to a
point in the exterior of the unit disk we can neglect $1/(J^{-1}(z))^{n}$ in
\[
T_{n}(z)=\hf\left((J^{-1}(z))^{n}+\nth{(J^{-1}(z))^{n}}\right).
\]
Thus, the function in (\ref{M ellipse}) is asymptotically equivalent to
(\ref{ChebM}). Theorem~\ref{THM ellipse} is now proved.

\medskip

\noindent\textbf{Acknowledgments.}
The authors wish to thank Mihai Putinar for
enlightening discussions during the BIRS workshop. We also thank the referees
for valuable suggestions and new references.
This material is based upon
work supported by the National Science Foundation under Grant No. DMS-1714287.

\appendix
\section{Appendix}

\begin{lemma} \label{LEM sum asymptotics}
Let $\{a_n, b_n\}_{n=1}^\infty$ be nonnegative numbers such that $a_n \simeq e^{-\alpha n}$ and $b_n \simeq e^{-\beta n}$ with $0< \beta < \alpha$, where the implicit constants don't depend on $n$. Let $\eta > 0$ be a small parameter, then

\begin{equation} \label{sum asymptotics}
\sum_{n=1}^\infty \frac{b_n}{a_n + \eta} \simeq \eta^{\frac{\beta}{\alpha} - 1},\qquad \text{and} \qquad
\sum_{n=1}^\infty \frac{b_n}{(a_n + \eta)^2} \simeq \eta^{\frac{\beta}{\alpha} - 2}
\end{equation}

\noindent where the implicit constants don't depend on $\eta$.

\end{lemma}

\begin{proof}
Let us prove the first assertion of \eqref{sum asymptotics}, the second one will follow analogously. Introduce the switchover index $J=J(\eta) \in \NN$ defined by

\begin{equation*}
\begin{cases}
a_n \geq \eta & \quad \forall \ 1\leq n \leq J
\\
a_n < \eta & \quad \forall \ n> J
\end{cases}
\end{equation*}

\noindent Below all the implicit constants in relations involving $\simeq$ or $\lesssim$ will be independent on $\eta$. It is clear that

\begin{equation*}
\sum_{n=1}^\infty \frac{b_n}{a_n + \eta} \simeq \sum_{n \leq J} \frac{b_n}{a_n} + \frac{1}{\eta} \sum_{n>J} b_n.
\end{equation*}

\noindent Note that

\begin{equation*}
\sum_{n>J} b_n \lesssim \sum_{n>J} e^{-\beta n} \lesssim e^{-\beta (J+1)},
\end{equation*}

\noindent therefore using our assumption on $b_n$ we find

\begin{equation} \label{second sum}
\sum_{n>J} b_n \simeq b_{J+1} \simeq b_J.
\end{equation}

\noindent On the other hand

\begin{equation*}
\sum_{n \leq J} \frac{b_n}{a_n} \lesssim \sum_{n \leq J} e^{(\alpha - \beta) n} = \frac{e^\alpha}{e^\alpha - e^\beta} (e^{(\alpha - \beta) J} - 1) \lesssim e^{(\alpha - \beta) J} \simeq \frac{b_J}{a_J},
\end{equation*}

\noindent Thus we conclude

\begin{equation} \label{first sum}
\sum_{n \leq J} \frac{b_n}{a_n} \simeq \frac{b_J}{a_J}.
\end{equation}

\noindent Now $\eta \simeq a_J$ and $a_J \simeq e^{-\alpha J}$, therefore $e^{-J} \simeq \eta^{\frac{1}{\alpha}}$. Using these along with \eqref{second sum} and \eqref{first sum} we obtain   

\begin{equation*}
\sum_{n=1}^\infty \frac{b_n}{a_n + \eta} \simeq \frac{b_J}{a_J} + \frac{b_J}{\eta} \simeq \frac{b_J}{a_J} \simeq e^{(\alpha- \beta) J} \simeq \eta^{\frac{\beta}{\alpha} - 1}.
\end{equation*}

\end{proof}

\bibliographystyle{abbrv}
\bibliography{refs}
\end{document}